\pdfoutput=1
\RequirePackage{ifpdf}
\ifpdf 
\documentclass[pdftex]{sigma}
\else
\documentclass{sigma}
\fi

\usepackage{tikz}

\numberwithin{equation}{section}

\newtheorem{Theorem}{Theorem}[section]
\newtheorem{Corollary}[Theorem]{Corollary}
\newtheorem{Lemma}[Theorem]{Lemma}

\newtheorem{Claim}[Theorem]{Claim}
 { \theoremstyle{definition}
\newtheorem{Example}[Theorem]{Example}
\newtheorem{Examples}[Theorem]{Examples}
 }

\newcommand{\diag}{\operatorname{diag}}

\newcommand{\Fl}{\operatorname{\CF\ell}}

\newcommand{\Gr}{\operatorname{Gr}}

\newcommand{\Mat}{\operatorname{Mat}}
\newcommand{\maxx}{{\operatorname{max}}}
\newcommand{\minn}{{\operatorname{min}}}

\newcommand{\Pl}{\operatorname{\CP\ell}}

\newcommand{\rank}{\operatorname{rank}}

\newcommand{\hra}{\hookrightarrow}
\newcommand{\iso}{\overset{\sim}{\longrightarrow}}
\newcommand{\isom}{\overset{\sim}{=}}
\newcommand{\lra}{\longrightarrow}
\newcommand{\ra}{\rightarrow}

\newcommand{\bag}{{\bar{g}}}
\newcommand{\bap}{{\bar{p}}}

\newcommand{\bax}{{\bar{x}}}

\newcommand{\tFl}{{\tilde{\CF\ell}}}

\newcommand{\tPl}{{\tilde{\CP\ell}}}

\newcommand{\ty}{\tilde y}

\newcommand{\fb}{\mathfrak b}
\newcommand{\ff}{\mathfrak f}

\newcommand{\fh}{\mathfrak h}
\newcommand{\fl}{\mathfrak l}
\newcommand{\fm}{\mathfrak m}

\newcommand{\fgl}{\mathfrak{gl}}

\newcommand{\fW}{\mathfrak W}

\newcommand{\bno}{\mathbf{0}}

\newcommand{\bk}{\mathbf{k}}

\newcommand{\CB}{\mathcal{B}}
\newcommand{\CC}{\mathcal{C}}

\newcommand{\CF}{\mathcal{F}}

\newcommand{\CP}{\mathcal{P}}

\newcommand{\CW}{\mathcal{W}}

\newcommand{\BA}{\mathbb{A}}
\newcommand{\BC}{\mathbb{C}}

\newcommand{\BF}{\mathbb{F}}

\newcommand{\BN}{\mathbb{N}}
\newcommand{\BP}{\mathbb{P}}
\newcommand{\BQ}{\mathbb{Q}}

\newcommand{\BZ}{\mathbb{Z}}

\begin{document}

\newcommand{\arXivNumber}{2007.04045}

\renewcommand{\thefootnote}{}

\renewcommand{\PaperNumber}{001}

\FirstPageHeading

\ShortArticleName{With Wronskian through the Looking Glass}

\ArticleName{With Wronskian through the Looking Glass\footnote{This paper is a~contribution to the Special Issue on Representation Theory and Integrable Systems in honor of Vitaly Tarasov on the 60th birthday and Alexander Varchenko on the 70th birthday. The full collection is available at \href{https://www.emis.de/journals/SIGMA/Tarasov-Varchenko.html}{https://www.emis.de/journals/SIGMA/Tarasov-Varchenko.html}}}

\Author{Vassily GORBOUNOV~$^{\dag\ddag}$ and Vadim SCHECHTMAN~$^\S$}

\AuthorNameForHeading{V.~Gorbounov and V.~Schechtman}

\Address{$^\dag$~HSE University, Russia}
\Address{$^\ddag$~Laboratory of Algebraic Geometry and Homological Algebra, \\
\hphantom{$^\ddag$}~Moscow Institute of Physics and Technology, Dolgoprudny, Russia}
\EmailD{\href{mailto:vgorb10@gmail.com}{vgorb10@gmail.com}}

\Address{$^\S$~Institut de Math\'ematiques de Toulouse, Universit\'e Paul Sabatier,\\
\hphantom{$^\S$}~118 Route de Narbonne, 31062 Toulouse, France}
\EmailD{\href{mailto:schechtman@math.ups-tlse.fr}{schechtman@math.ups-tlse.fr}}

\ArticleDates{Received September 01, 2020, in final form December 27, 2020; Published online January 02, 2021}

\Abstract{In the work of Mukhin and Varchenko from 2002 there was introduced a Wronskian map from the variety of full flags in a finite dimensional vector space into a product of projective spaces. We establish a precise relationship between this map and the Pl\"ucker map. This allows us to recover the result of Varchenko and Wright saying that the polynomials appearing in the image of the Wronsky map are the initial values of the tau-functions for the Kadomtsev--Petviashvili hierarchy.}


\Classification{37K20; 81R10; 35C08}

\begin{flushright}
\begin{minipage}{75mm}
\it To Vitaly Tarasov and Alexander Varchenko,\\
as a token of friendship
\end{minipage}
\end{flushright}

\renewcommand{\thefootnote}{\arabic{footnote}}
\setcounter{footnote}{0}

\rightline{\includegraphics{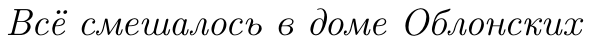}}

\section{Introduction}

Let $G = {\rm GL}_n(\BC)$, $T\subset B_- \subset G$ the subgroups of diagonal and lower triangular
matrices, $X = B_-\backslash G$ the variety of full flags in $V = \BC^n$.

We have the {\it Pl\"ucker embedding}
\begin{gather*}
\Pl = (\Pl_1, \dots, \Pl_{n-1})\colon \ X\hra \BP := \prod_{i=1}^{n-1} \BP^{C_n^i-1},
\end{gather*}
where $C_n^i = \binom{n}{i}$.
On the other hand, in \cite{MV} there was introduced a map
\[
\fW\colon \ X \lra \big(\BP^N\big)^{n-1}
\]
($N$ being big enough), which we call the {\em Wronskian map} since its definition uses a lot of Wronskians. This map has been studied in~\cite{SV}. We will see below that $\fW$ lands in a subspace
\[
\prod_{i=1}^{n-1} \BP^{i(n-i)-1}\subset \big(\BP^N\big)^{n-1},
\]
so we will consider it as a map
\begin{gather}
\fW = (\fW_1, \dots, \fW_{n-1})\colon\ X \lra \BP' := \prod_{i=1}^{n-1} \BP^{i(n-i)-1}.
\label{eq1.2}
\end{gather}

The present note, which may be regarded as a postscript to \cite{SV}, contains some elementary remarks on the relationship between $\Pl$ and $\fW$.

We define for each $1\leq i\leq n - 1$ a linear {\it contraction map}
\[
c_i\colon \ \BP^{C_n^i - 1} \lra \BP^{i(n-i) - 1},
\]
such that
\[
\fW_i = c_i\circ \Pl_i,
\]
see Theorem~\ref{theorem3.3}.

For $g\in G$ let $\bag\in X$ denote its image in $X$; let
\[
\fW_i(g) = (a_0(g):\dots : a_{i(n-i)}(g)),
\]
and consider a polynomial
\[
y_i(g)(x) = \sum_{q=0}^{i(n-i)-1} a_q(g)\frac{x^q}{q!}.
\]
As a corollary of Theorem~\ref{theorem3.3} we deduce that the polynomials $y_i(g)$ are nothing else but the initial values of the tau-functions for the KP hierarchy, see Theorems~\ref{theorem4.4.2} and~\ref{theorem4.5.2}; this assertion is essentially \cite[Lemma~5.7]{VW}.

As another remark we reinterpret in Section~\ref{section5} the {\it $W5$ identity} instrumental in \cite{SV} as a~particular case of the classical Desnanot--Jacobi formula, and explain its relation to Wronskian mutations
studied in~\cite{MV,SV}.
We note some interesting related references \cite{EG,G}.

\section{Wronsky map}\label{section2}

We fix a base commutative ring $\bk\supset \BQ$.
Let
\[
\ff = (f_1(t), \dots, f_n(t))
\]
be a sequence of rational functions $f_i(t)\in \bk(t)$. Its Wronskian matrix
is by definition an $n\times n$ matrix
\[
\CW(\ff) = \big(f_i^{(j)}(t)\big),
\]
where $f^{(j)}(t)$ denotes the $j$-th derivative.

The determinant of $\CW(\ff)$ is
the {\it Wronskian} of $\ff$:
\[
W(\ff) = \det(\CW(\ff)).
\]

If $A = (a_{ij})\in \fgl_n(\bk)$ is a scalar matrix,
\begin{gather*}
W(\ff A) = \det(A)W(\ff),\\
W(f_1, \dots, f_n)' = \sum_{i=1}^n W(f_1, \dots, f'_i, \dots, f_n).
\end{gather*}

\subsection[The map W]{The map $\boldsymbol{\fW}$}\label{section2.2}
Let
\[
M = (b_{ij})_{1\leq i\leq n, 1\leq j\leq m}\in \Mat_{n,m}(\bk)
\]
be a rectangular matrix. Let us associate to it a sequence of polynomials of degree
$m - 1$
\[
\fb(M) = (b_1(M,t),\dots, b_n(M,t)),\qquad b_i(t) = \sum_{j=0}^{m-1} b_{i,j+1}\frac{t^j}{j!}.
\]
In general we identify the space of polynomials of degree $\leq m - 1$ with the space
of $\bk$-points of an affine space:
\begin{gather*}
\bk[t]_{\leq m - 1} \iso \BA^{m+1}(\bk),\qquad \sum_{j=0}^{m-1} b_{j}\frac{t^j}{j!} \mapsto
(b_0, \dots, b_{m-1}).
\end{gather*}

For $1 \leq j\leq n$ let $\fb_{\leq j}(M)$ denote the truncated sequence
\[
\fb(M)_{\leq j} = (b_1(M,t),\dots, b_j(M,t)).
\]
We define a sequence of polynomials
\begin{gather*}
\fW(M) = (y_1(M),\dots , y_{n-1}(M)) :=
(W(\fb(M)_{\leq 1}),\dots , W(\fb(M)_{\leq n - 1}))\in \bk[t]^{n - 1}.
\end{gather*}
Note that if $n = m$ then
\[
W(\fb(M)) = \det M,
\]
it is a constant polynomial.

If $A \in \fgl_n(k)$ then
\[
\fb(AM) = \fb(M)A^t.
\]
It follows that if $e_{ij}(a)$, $i > j$, is a lower triangular elementary matrix then
\[
\fb(e_{ij}(a)M) = (b_1(M), \dots, b_j(M) + ab_i(M), b_{j+1}(M), \dots),
\]
whence
\[
\fW(e_{ij}(a)M)) = \fW(M).
\]
It follows that for any $A\in N_-(\bk)$ (a lower triangular with $1$'s on the diagonal)
\begin{gather}
\fW(AM) = \fW(M).
\label{eq2.2.2}
\end{gather}

On the other hand, if
\[
D = \diag(d_1, \dots, d_n)\in \fgl_n(\bk)
\]
then
\[
\fW(DM) = \prod_{i=1}^n d_i\cdot \fW(M).
\]
In other words, if $A\in B_-(n,\bk)$ (the lower Borel),{\samepage
\begin{gather}
\fW(AM) = \det(A)\fW(M),
\label{eq2.2.3}
\end{gather}
which of course is seen immediately.}

\subsection{Degrees and Bruhat decomposition}\label{section2.3}

Suppose that $n = m$. We will denote by $G = {\rm GL}_n$, $B_-\subset G$ the lower triangular Borel, $N_-\subset B_-$ etc.

For a matrix $g\in G(\bk)$ let $\fW(g) = (y_1(g), \dots , y_{n-1}(g))$,
and consider the vector of degrees
\[
d(g) = (d_1(g), \dots d_{n-1}(g)) = (\deg y_1(g), \dots , \deg y_{n-1}(g))\in \BN^{n-1}.
\]
It turns out that $d(g)$ can take only $n!$ possible values
situated in vertices of a permutohedron.

Namely, consider the Bruhat decomposition
\[
G(\bk) = \cup_{w\in W} B_-(\bk)wB_-(\bk),
\]
$W = S_{n-1} = W(G, T)$ being the Weyl group.

Identify $\BN^{n-1}$ with the root lattice $Q$ of $G^s := {\rm SL}_n$ using the standard base \mbox{$\{\alpha_1, \dots, \alpha_{n-1}\} \!\subset\! Q$} of simple roots.

Then it follows from \cite[Theorem 3.12]{MV}, that for $g\in B(w) := B_-(\bk)wB_-(\bk)$
\begin{gather*}
d(g) = w * \bno,
\end{gather*}
where $\bno = (0, \dots, 0)$ and $*$ denotes the usual shifted Weyl group action
\[
w*\alpha = w(\alpha - \rho) + \rho.
\]
In other words, if
\[
w * \bno = \sum_{i=1}^{n-1} d_i(w)\alpha_i
\]
then
\[
d_i(g) = d_i(w),\qquad 1\leq i\leq n - 1.
\]

\begin{Example}\label{example2.3.1}
Let $n = 3$. For $g = (a_{ij})\in {\rm GL}_3(\bk)$
\begin{gather*}
y_1(g) = a_{11} + a_{12}x + a_{13}\frac{x^2}{2},
\\
y_2(g) = \Delta_{11}(g) + \Delta_{13}(g)x + \Delta_{23}(g)\frac{x^2}{2}.
\end{gather*}
Here $\Delta_{ij}(g)$ denotes the $2\times 2$ minor of $g$ picking the first two rows
and $i$-th and $j$-th columns.

We have two simple roots $\alpha_1$, $\alpha_2$.

For $g\in {\rm GL}_3(\bk)$ the vector $d(g)$ can take $6$ possible values:
$(0,0)$, $(1,0)$, $(0,1)$, $(1,2)$, $(2,1)$, and $(2,2)$, these vectors forming a hexagon, cf.~\cite[Section~3.5]{MV}.

One checks directly that
\[
B_- = d^{-1}(0,0).
\]
This means that
$g\in B_-$ if and only if $a_{12} = a_{13} = \Delta_{23}(g) = 0$.

Similarly
\begin{gather*}
B_-(12)B_- = d^{-1}(1,0),\qquad B_-(23)B_- = d^{-1}(0,1),\qquad
B_-(123)B_- = d^{-1}(2,1),
\end{gather*}
i.e.,
$g\in B_-(123)B_-$ if and only if $\Delta_{23}(g)= 0$;
\begin{gather*}
 B_-(132)B_- = d^{-1}(2,1),
\qquad
B_-(13)B_- = d^{-1}(2,2)
\end{gather*}
(the big cell). This means that
$g\in B_-(13)B_-$ if and only if $a_{12} \neq 0$, $a_{13} \neq 0$, $\Delta_{23}(g)\neq 0$.

These formulas may be understood as a criterion for recognizing the Bruhat cells in ${\rm GL}_3$, cf.~\cite{FZ}.
\end{Example}

Therefore the Wronskian map induces maps
\begin{gather*}
\fW(w)\colon \ B(w) \lra \prod_{i=1}^{n-1} \BP^{d_i(w)}
\end{gather*}
for each $w\in W$.

\subsection{Induced map on the flag space}\label{section2.4}
Let $D\in \BN$ be such that $d_i(g) \leq D$ for all
$g\in G$, $i\in [n-1] := \{ 1, \dots, n-1\}$.

The invariance \eqref{eq2.2.2} implies that $\fW$ induces a map from the {\it base affine space}
\begin{gather*}
\fW_{\tFl_-}\colon \ \tFl_- := N_-\backslash G \lra \bk[t]_{\leq D}^{n-1} \isom \big(\BA^{D+1}\big)^{n-1}(\bk),
\end{gather*}
while \eqref{eq2.2.3} implies that $\fW$ induces a map from the full flag space
\begin{gather*}
\fW_{\Fl_-}\colon \ \Fl_- := B_-\backslash G \lra \BP(\bk[t])_{\leq D}^{n-1}\isom \big(\BP^{D}\big)^{n-1}(\bk).
\end{gather*}

More explicitly:
we can assign to an arbitrary matrix $g = (b_{ij})\in G$ a flag in $V = \bk^n$
\[
F(g) = V_1(g) \subset \dots \subset V_n(g) = V,
\]
whose $i$-th space $V_i(g)$ is spanned by the first $i$ row vectors of $g$
\[
v_j(g) = (b_{j1}, \dots , b_{jn})\in V,\qquad 1\leq j \leq i.
\]
It is clear that $F(g) = F(ng)$ for $n\in B_-$, and the map
\begin{gather*}
F\colon \ G \lra \Fl(V),
\end{gather*}
where $\Fl(V)$ is the space of full flags in $V$,
induces an isomorphism
\[
B_-\backslash G \iso \Fl(V).
\]
On the other hand consider the restriction of $F$ to the upper triangular group
\begin{gather*}
F_N\colon \ N \hra \Fl(V);
\end{gather*}
this map is injective and its image is the big Schubert cell.

We may also consider the composition
\begin{gather*}
\fW_N\colon \ N\hra \Fl_- \overset{\fW_{\Fl_-}}\lra \big(\BP^D\big)^{n-1}.
\end{gather*}
We will see below (cf.\ Section~\ref{section3.8}) that this map is an embedding.

\subsection{Partial flags}\label{section2.5}
More generally, for any unordered partition
\[
\lambda\colon \ n = n_1 + \dots + n_p,\qquad n_i\in \BZ_{\BZ > 0}
\]
we define in a similar way a map
\begin{gather*}
\fW\colon \ \Fl_{\lambda, -} := P_{\lambda,-}\backslash G \lra \big(\BP^D\big)^{p-1}.
\end{gather*}
For example, for $p = 2$ (Grassmanian case) corresponding to a partition
$\lambda = i + (n - i)$
\[
\fW\colon \ \Fl_{\lambda, -} = \Gr_n^i \lra \BP^D.
\]

\section{Coefficients of Wronskians}\label{section3}

\subsection{Pl\"ucker map}\label{section3.1}
\subsubsection{Schubert cells in a Grassmanian}\label{section3.1a}
For $i\in [n] : = \{1, \dots, n\}$ let $\CC_n^i$ denote the set of all $i$-element subsets of $[n]$.

The natural action of $W = S_n$ on $\CC_n^i$ identifies
\begin{gather*}
\CC_n^i \isom S_n/S_i.
\end{gather*}
Let $P_i\subset G = {\rm GL}_n$ denote the stabilizer of the coordinate subspace $\BA^i\subset \BA^n$, so that
\[
G/P_i \isom \Gr^i_n,
\]
the Grassmanian of $i$-planes in $\BA^n$.
The Bruhat lemma gives rise to an isomorphism
\begin{gather*}
\CC_n^i \isom B\backslash G/P_i.
\end{gather*}
This set may also be interpreted as ``the set of $\BF_1$-points''
\begin{gather*}
\CC_n^i = \Gr_n^i(\BF_1),
\end{gather*}
whose cardinality is a binomial coefficient
\[
|\CC_n^i| = C_n^i = \binom{n}{i}.
\]

\subsubsection{A Pl\"ucker map}\label{section3.1b} Consider a matrix
\[
M = (b_{ij})_{i\in [n], j\in [m]} \in \Mat_{n,m}(\bk)
\]
with $n\leq m$. For any $j\in [n]$ $M_{\leq j}$ will denote the truncated matrix
\[
M_{\leq j} = (b_{ip})_{i\in [j], p\in [m]} \in \Mat_{j,m}(\bk).
\]
We suppose that $\rank(M) = n$.

For any $j\in [n]$ consider the set of $j\times j$ minors of $M$
\[
p_j(M) = (\Delta_{[j], I})_{I\in \CC_m^j}\in \BA^{C_m^j}(\bk)
\]
or the same set up to a multiplication by a scalar
\[
\bap_j(M) = \pi(p_j(M)) = \BP^{C_m^j - 1}(\bk),
\]
where $\pi\colon \BA^{C_m^j}(\bk) \setminus \{\bno\} \lra \BP^{C_m^j - 1}(\bk)$ is the canonical projection.

We will use notations
\[
\tPl(M) = (p_1(M), \dots, p_m(M))\in \prod_{j=1}^m \BA^{C_m^j}(\bk)
\]
and
\[
\Pl(M) = (\bap_1(M), \dots, \bap_m(M))\in \prod_{j=1}^m \BP^{C_m^j - 1}(\bk).
\]
Suppose that $m = n$, so we get maps
\[
\tPl = (\Pl_1, \dots, \Pl_n)\colon \ {\rm GL}_n \lra \prod_{j=1}^n \BA^{C_n^j}
\]
and
\[
\Pl\colon \ {\rm GL}_n \lra \prod_{j=1}^n \BP^{C_n^j - 1}.
\]
It is clear that $\Pl(zM) = \Pl(M)$ for $z\in N_-\subset G = {\rm GL}_n$ and
$\Pl(bM) = \Pl(M)$ for $b\in B_-\subset G$, so $\tPl$, $\Pl$ induce maps
\begin{gather*}
\Pl_{\tFl}\colon \ \tFl_- = N_-\backslash G \lra \prod_{j=1}^n \BA^{C_n^j}
\end{gather*}
and
\begin{gather*}
\Pl_{\Fl}\colon \ \Fl_- = B_-\backslash G \lra \prod_{j=1}^n \BP^{C_n^j - 1}.
\end{gather*}

\subsection{Schubert cells in Grassmanians}\label{section3.2}
 For $i\in [n]$ consider the $i$-th Pl\"ucker map
\[
\Pl_i\colon \ G \lra \BA^{C_n^i}.
\]
We will compare it with the $i$-th component of the Wronskian map
\[
\fW_i\colon \ G \lra \BA^{d_i+1},\qquad g\mapsto y_i(g).
\]
To formulate the result we will use the Schubert decomposition from Section~\ref{section3.1a}.

Let $p = \ell(w)$ denote the length of a minimal
decomposition
\[
w = s_{j_1}\cdots s_{j_p}
\]
into a product of Coxeter generators $s_j = (j, j+1)$.

Let
\[
I_0 = I_\minn = \{1, \dots, i\}\in \CC_n^i.
\]
Sometimes it is convenient to depict elements of $\CC_n^i$ as sequences
\begin{gather}
I = (e_1\dots e_n),\qquad e_j \in \{0,1\},\qquad \sum e_j = i.
\label{eq3.2.1}
\end{gather}

In this notation
\[
I_0 = 1\dots 10\dots 0.
\]

We define a length map
\begin{gather}\label{eq3.2.1a}
\fl\colon \ \CC_n^i \lra \BZ_{\geq 0},
\end{gather}
as follows: identify $\CC_n^i\isom S_n/S_i$, then for $\bax\in \CC_n^i$ $\ell(\bax)$ is the minimal length of a representative $x\in S_n$.

\begin{Example}\label{example3.2.2} $\fl(I_0) = 0$. The element of maximal length is
\[
I_{\max} = 0\dots 01\dots 1.
\]
Its length is
\[
\fl(I_{\max}) = i(n-i).
\]
\end{Example}

\begin{Claim}\label{claim3.2.3}
Let $w_0\in S_n = W({\rm GL}_n)$ denote the element of maximal length. Then
\[
d_i := d(w_0) = \fl(I_{\max}) = i(n-i).
\]
\end{Claim}

This is a particular case of a more general statement, see below Corollary~\ref{corollary3.5.2}.

For each $j\geq 0$ consider the subset
\[
\CC_n^i(j) = \fl^{-1}(j) \subset \CC_n^i,
\]
so that
\[
\CC_n^i = \coprod_{j=0}^{d_i} \CC_n^i(j).
\]

For example
\[
\CC_n^i(0) = \{I_0\},\qquad \CC_n^i(d_i) = \{I_\maxx\}.
\]

\subsubsection{Symmetry}\label{section3.2.4}
\[
|\CC_n^i(j)| = |\CC_n^i(d_i - j)|.
\]

\subsubsection{Range}\label{section3.2.5}
$\CC_n^i(j)\neq \varnothing$ iff $0\leq j\leq d_i$. In other words, the range of $\fl$ is $\{0, \dots, d_i\}$.

The following statement is the main result of the present note.

\subsection[From Pl to W: a contraction]{From $\boldsymbol{\Pl}$ to $\boldsymbol{\fW}$: a contraction}\label{section3.3}

\begin{Theorem}\label{theorem3.3}
\begin{gather*}
y_i(g)(t) = \sum_{I\in \CC_n^i}\Delta_{[i],I}(g)m(I)\frac{t^{\fl(I)}}{\fl(I)!}
= \sum_{j=0}^{d_i} \biggl(\sum_{I\in \CC_n^i(j)} m(I)\Delta_{[i],I}(g)\biggr)\frac{t^j}{j!},
\end{gather*}
where the numbers $m(I)\in \BZ_{>0}$ are defined below, see Section~{\rm \ref{section3.5}}.

In other words, the map $\fl$ induces a contraction map
\begin{gather*}
c = c_i\colon \ \BA^{\CC_n^i} \lra \bk[t]_{\leq d_i},\\
c((a_I)_{I\in \CC_n^i}) = \sum_{j=0}^{d_i}\biggl(\sum_{I\in \fl^{-1}(j)} m(I)a_I\biggr)\frac{t^j}{j!}.
\end{gather*}
Then
\[
\fW_i = c_i\circ \Pl_i.
\]
\end{Theorem}

We will also use the reciprocal polynomials
\begin{gather}
\ty_i(g)(x) = x^{d_i}y_i(g)\big(x^{-1}\big).
\label{eq3.3.1}
\end{gather}

Proof of Theorem~\ref{theorem3.3} is given below 
after some preparation.

\begin{Examples}\label{examples3.4}\quad
\begin{enumerate}\itemsep=0pt
\item[(i)] Let $n = 4$, $i = 2$. A formula for $y_2(g)$, $g\in N_4$ is given in \cite[equation~(5.11)]{SV}:
\begin{gather}
y_2(g)(t) = \Delta_{12}(g) + \Delta_{13}(g)t + (\Delta_{14}(g) + \Delta_{23}(g))
\frac{t^2}{2} + 2\Delta_{24}(g)\frac{t^3}{6} + 2\Delta_{34}(g)\frac{t^4}{24},\!\!\!
\label{eq3.4.1}
\end{gather}
where for brevity
\[
\Delta_I := \Delta_{12,I}
\]
for $I\subset [4]$.

\item[(ii)] More generally, for $g\in {\rm GL}_n$
\begin{gather*}
y_2(g) = \Delta_{12}(g) + \Delta_{13}(g)t + \big(\Delta_{14}(g) + \Delta_{23}(g)\big) \frac{t^2}{2} \\
\hphantom{y_2(g) =}{}
+ \big(2\Delta_{24}(g) + \Delta_{15}(g)\big) \frac{t^3}{6} +
\big(2\Delta_{34}(g) + 2\Delta_{25}(g) + \Delta_{16}(g)\big)\frac{t^4}{24}+ \cdots,
\end{gather*}
$\deg y_2(g)\leq n(n-2)$, the exact degree depends on the Bruhat cell which $g$ belongs to.

\item[(iii)] Let again $n = 4$. Then (see \cite[equation~(5.11)]{SV})
\[
y_3(g) = \Delta_{123}(g) + \Delta_{124}(g)t + \Delta_{134}(g)\frac{t^2}{2} + \Delta_{234}(g)\frac{t^3}{6}.
\]
\end{enumerate}
\end{Examples}

\subsection[Creation operators Delta i]{Creation operators $\boldsymbol{\Delta_i}$}\label{section3.5} Fix $n\geq 2$.
For
\[
I = \{i_1, \dots, i_k\}\in \CC_n^k
\]
we imply that $i_1 < \dots < i_k$.

We call $i = i_p$ {\it admissible} if either $p = k$ and $i_p < n$ or
$i_{p+1} > i_p + 1$. We denote by $I^o\subset I$ the subset of admissible elements.

For each $i = i_p\in I^o$ we define a new set
\[
\Delta_iI = \{i'_1, \dots, i'_k\},
\]
where $i'_q = i_q$ if $q \neq p$, and $i'_p = i_p + 1$.

The reader should compare this definition with operators defining a representation of
the {\it nil-Temperley--Lieb algebra} from \cite[equation~(2.4.6)]{BFZ}, cf.\ also~\cite{BM} and references therein.

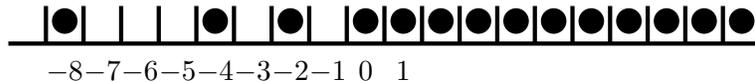
\begin{figure}[h!]\centering
\begin{tikzpicture}
\draw[ultra thick] (0,0) -- (10,0); 
\draw[ultra thick] (2.5,0) -- (2.5,0.5);
\foreach \x in {0.5,1,...,2} \draw[ultra thick] (\x,0) --
(\x,0.5);
\foreach \x in {3,3.5,...,10} \draw[ultra thick] (\x,0) -- (\x,0.5);
\fill[black] (0.75,0.3) circle (0.17cm);
\fill[black] (2.75,0.3) circle (0.17cm);
\fill[black] (3.75,0.3) circle (0.17cm);
\draw (0.75,-0.1)
node[anchor=north] {$-8$};
\draw (1.25,-0.1)
node[anchor=north] {$-7$};
\draw (1.75,-0.1)
node[anchor=north] {$-6$};
\draw (2.25,-0.1)
node[anchor=north] {$-5$};
\draw (2.75,-0.1)
node[anchor=north] {$-4$};
\draw (3.25,-0.1)
node[anchor=north] {$-3$};
\draw (3.75,-0.1)
node[anchor=north] {$-2$};
\draw (4.25,-0.1)
node[anchor=north] {$-1$};
\draw (4.75,-0.1)
node[anchor=north] {$0$};
\draw (5.25,-0.1)
node[anchor=north] {$1$};
\foreach \x in {5,5.5,...,10} \fill[black] (\x-0.25,0.3) circle (0.17cm);
\end{tikzpicture}
\caption{``Balls in boxes'' picture.}
\end{figure}

Recall the representation \eqref{eq3.2.1} of elements of $\CC_n^k$:
where we imagine the $1$'s as $k$ ``balls'' sitting in $n$ ``boxes''.
An operation $\Delta_i$ means moving
the ball in $i$-th box to the right, which is possible if the $(i+1)$-th box is free.

Each $I\in \CC_n^k$ may be written as{\samepage
\begin{gather}
I = \Delta_{j_p}\cdots \Delta_{j_1}I_\minn
\label{eq3.5.1}
\end{gather}
for some $j_1, \dots, j_p$. This is clear from the balls in boxes picture.}

Let $\fm(I)$ denote the set of all sequences
$j_1, \dots, j_p$ such that \eqref{eq3.5.1} holds, and
\[
m(I) := |\fm(I)|.
\]

\begin{Claim}\label{claim3.5.1}
The lengths $p$ of all sequences $(j_1, \dots, j_p)\in \fm(I)$ are the same, namely $p = \fl(I)$.
Here $\fl(I)$ is from \eqref{eq3.2.1a}.
\end{Claim}

\begin{proof} Clear from balls in boxes picture.\end{proof}

\begin{Corollary}\label{corollary3.5.2}
Let
\[
I = \{a_1, \dots, a_k\},
\]
then
\[
\fl(I) = \sum_{i=1}^k (a_i - i).
\]
\end{Corollary}

To put it differently, define a graph $\Gamma_n^k$ whose set of vertices is $\CC_n^k$, the edges having the form
\[
I \lra \Delta_iI
\]
(or otherwise define an obvious partial order on $\CC_n^i$). Then $\fm(I)$ is the set
of paths in $\Gamma_n^k$ going from the minimal element $[k]$ to $I$.

\subsubsection{Symmetry}\label{section3.5.3}
{\it This graph can be turned upside down.}
Clear from the ``balls in boxes'' description.

\subsection{Generalized Wronskians and the derivative}\label{section3.6} Let
\[
\ff = (f_1(t), f_2(t), \dots )
\]
be a sequence of functions. We can assign to it a $\BZ_{\geq 1}\times \BZ_{\geq 1}$ Wronskian matrix
\[
\CW(\ff) = \big(f_i^{(j-1)}\big)_{i, j\geq 1}.
\]
For each $I = \{i_1, \dots, i_k\} \in \CC_\infty^k$ let $\CW_I(\ff)$ denote a $(k\times k)$-minor of
$\CW(\ff)$ with rows $i_1, \dots, i_k$ and columns $1, 2, \dots, k$, and let
\[
W_{I}(\ff) = \det \CW_I(\ff).
\]

\begin{Lemma}\label{lemma3.6.1} The derivative
\[
W_I(\ff)' = \sum_{i\in I^o} W_{\Delta_iI}(\ff).
\]
\end{Lemma}

\begin{Corollary}\label{corollary3.6.2}
\[
W_I(\ff)^{(p)} = \sum_{(i_1,\dots, i_p)\ \text{composable}}W_{\Delta_{i_p}\cdots\Delta_{i_1}I}(\ff),
\]
where a sequence $(i_1,\dots, i_p)$ is called composable if
$i_q\in (\Delta_{i_{q-1}}\cdots\Delta_{i_1}I)^o$ for all $q$.
\end{Corollary}

\begin{proof}[Proof of Theorem~\ref{theorem3.3}]
We shall use a formula:
if
\[
y(t) = \sum_{i\geq 0} a_i\frac{t^i}{i!}
\]
then
\[
a_i = y^{(i)}(0).
\]
Let $g = (b_{ij})\in {\rm GL}_n$,
\[
b_i(t) = \sum_{j=0}^{n-1} b_{i,j+1}\frac{t^j}{j!},\qquad 1\leq i\leq n.
\]
By definition
\[
y_i(g)(t) = W(b_1(t), \dots, b_i(t)),\qquad 1\leq i\leq n.
\]
For $I, J\subset [n]$ let $M_{IJ}(g)$ denote the submatrix of $g$ lying on the intersection
of the lines (columns) with numbers $i\in I$ ($j\in J)$, so that
\[
\Delta_{IJ}(g) = \det M_{IJ}(g).
\]
We see that the constant term
\[
y_i(g)(0) = \det \big(M_{[i],[i]}(g)^t\big) = \Delta_{[i]}(g),
\]
where $M^t$ denotes the transposed matrix.

To compute the other coefficients
we use Lemma~\ref{lemma3.6.1} and Corollary~\ref{corollary3.6.2}. So
\[
y'_i(0) = \det\big(M_{[i],\Delta_i[i]}(g)^t\big) = \Delta_{[i],\Delta_i[i]}(g),
\]
and more generally
\begin{gather*}
y^{(p)}_i(0) = \sum_{(i_1, \dots, i_p)\ \text{composing}}\det\big(
M_{[i],\Delta_{i_p}\cdots \Delta_{i_1}[i]}(g)^t\big)\\
\hphantom{y^{(p)}_i(0)}{}
= \sum_{(i_1, \dots, i_p)\ \text{composing}} \Delta_{[i],\Delta_{i_p}\cdots \Delta_{i_1}[i]}(g),
\end{gather*}
which implies the formula.
\end{proof}

\subsection{Triangular theorem}\label{section3.8}
Consider an upper triangular unipotent matrix
$g\in N\subset {\rm GL}_n(\bk)$.
We claim that $g$ may be reconstructed uniquely from the coefficients of
polynomials $y_1(g),\dots, y_{n-1}(g)$.

More precisely, to get the first $i$ rows of $g$ we need only a truncated part of the first
$i$ polynomials
\[
(y_1(g) = y_1(g)_{\leq n-1}, y_2(g)_{\leq n-2},\dots, y_i(g)_{\leq i}).
\]
This is the contents of \cite[Theorem~5.3]{SV}. We explain how it follows from our Theorem~\ref{theorem3.3}.

To illustrate what is going on consider an example $n = 5$. Let{\samepage
\[
g = \left(\begin{matrix} 1 & a_1 & a_2 & a_3 & a_4 \\
0 & 1 & b_2 & b_3 & b_4\\
0 & 0 & 1 & c_3 & c_4\\
0 & 0 & 0 & 1 & d_4\\
0 & 0 & 0 & 0 & 1
\end{matrix}\right).
\]
We have $y_1(g) = b_1(g)$, so we get the first row of $g$, i.e., the elements $a_i$, from $y_1(g)$.}

Next,
\begin{gather*}
y_2(g) = \Delta_{12}(g) + \Delta_{13}(g)x + (\Delta_{14}(g) +
\Delta_{23}(g))\frac{x^2}{2} + (\Delta_{15}(g) + \cdots )\frac{x^3}{6} + \cdots \\
\hphantom{y_2(g)}{}= 1 + b_2x + (b_3 + a_1b_2 - a_2)\frac{x^2}{2} + (b_4 + \cdots )\frac{x^6}{6} +
\cdots,
\end{gather*}
whence we recover $b_2$, $b_3$, $b_4$ (in this order) from $y_2(g)$, the numbers $a_i$ being already known.

Next,
\begin{gather*}
y_3(g) = \Delta_{123}(g) + \Delta_{124}(g)x + (\Delta_{134}(g) +\Delta_{125}(g))\frac{x^2}{2} + \cdots \\
\hphantom{y_3(g)}{} = 1 + c_3x + (c_4 + b_2c_3 - b_3)\frac{x^2}{2} + \cdots,
\end{gather*}
whence we recover $c_3$, $c_4$ (in this order) from $y_3(g)$.

Finally
\[
y_4(g) = \Delta_{1234}(g) + \Delta_{1235}(g)x + \cdots = 1 + d_4x + \cdots,
\]
whence $d_4$ from $y_4(g)$.

{\it Triangular structure on the map $\fW_N$.}
We can express the above as follows. Let $\CB := \fW(N)$, so that
\[
\fW_N\colon \ N\lra \CB.
\]
Obviously $N\isom \bk^{n(n-1)/2}$;
we define $n(n-1)/2$ coordinates in $N$ as the elements of a matrix $g\in N$ in the lexicographic order, i.e., $n-1$ elements from the first row (from left to right), $n-2$ elements from the second row, etc.

Let
\[
\fW(g) = (y_1(g), \dots, y_{n-1}(g));
\]
we define the coordinates of a vector $\fW(g)$ similarly, by taking $n-1$ coefficients
of $y_1(g)$, then the first $n-2$ coefficients of $y_2(g)$, etc.

\begin{Claim}\label{claim3.8.1}
The above rule defines a global coordinate system on
$\CB$, i.e., an isomorphism
\[
\CB\isom \bk^{n(n-1)/2},
\]
and the matrix of $\fW_N$ with respect to the above two lexicographic coordinate systems
is triangular with $1$'s on the diagonal.
\end{Claim}

\begin{Corollary}\label{corollary3.8.2}
The map $\fW$ \eqref{eq1.2} is an embedding.
\end{Corollary}

Indeed, $X$ is a union of open Schubert cells, the map $\fW$ is ${\rm GL}_n$-equivariant,
and each open cell may be transfered to $N$ using an appropriate $g\in {\rm GL}_n$.\footnote{We owe this remark to
A.~Kuznetsov.}

\section{Wronskians and tau-functions}\label{section4}

\subsection{Minors of the unit Wronskian}\label{section4.1}
Consider a $n\times n$ Wronskian matrix
\begin{gather*}
\CW_n(x) = \CW\big(1, x, x^2/2, \dots, x^{n-1}/(n-1)!\big).
\end{gather*}

\begin{Claim}\label{claim4.1.1}
For each $1\leq i\leq n$ and $I\in \CC^i_n$ we have
\[
\Delta_{[i], I}(\CW_n(x)) = \frac{x^{\fl(I)}}{n(I)}
\]
for some $n(I)\in \BZ_{> 0}$.
\end{Claim}

The exact value of $n(I)$ will be given below, see Claim~\ref{claim4.3.1}.
\begin{Example}\label{example4.1.2}
\[
\Delta_{14}(\CW_4(x)) = \frac{x^2}{2}.
\]
\end{Example}

\subsection[Schur functions, and embedding of a finite Grassmanian into the semi-infinite one]{Schur functions, and embedding of a finite Grassmanian\\ into the semi-infinite one}\label{section4.2}

 Recall one of possible definitions for Schur functions, cf.\ \cite[Section~8]{SW}. Let
\[
\nu = (\nu_0, \nu_1, \dots)
\]
be a {\em partition}, i.e., $\nu_i \in \BZ_{\geq 0}$, $\nu_i \geq \nu_{i-1}$ and there exists
$r$ such that $\nu_i = 0$ for $i\geq r$.

A Schur function
\[
s_\nu(h_1, h_2, \dots ) = \det (h_{\nu_i - i + j})_{i,j = 0}^{r-1},
\]
where the convention is $h_0 = 1$, $h_i = 0$ for $i < 0$,
cf.\ the {\it Jacobi--Trudi} formula~\cite[Chapter~I, equation~(3.4)]{M}.

\begin{Examples}\label{examples4.2.1}
\begin{gather*}
s_{11} = \left|\begin{matrix}
h_1 & h_2\\
1 & h_1
\end{matrix}\right| = h_1^2 - h_2,
\\
s_{111} = \left|\begin{matrix}
h_1 & h_2 & h_3\\
1 & h_1 & h_2\\
0 & 1 & h_1
\end{matrix}\right| = h_1^3 - 2h_1h_2 + h_3.
\end{gather*}
\end{Examples}

\subsubsection{Electrons and holes}
Let $\CC_{\infty/2}$ denote the set of subsets
\[
S = \{a_0, a_1, \dots, \}\subset \BZ
\]
such that both sets
$S\setminus \BN$, $\BN\setminus S$ are finite, and there exists $d\in \BZ$ such that $a_i = i - d$ for $i$ sufficiently large.

The number $d = d(S)$ is called the {\it virtual dimension} of $S$.

The elements of $\CC_{\infty/2}$ enumerate the cells of the {\it semi-infinite Grassmanian}, cf.~\cite{SW}.

We can imagine such $S$ as the set of boxes numbered by $i\in \BZ$, with balls put
into the boxes with numbers $a_0$, $a_1$, etc.

We can define the virtual dimension also as
\[
d(S) = |S\setminus \BN| - |\BN\setminus S|.
\]
Set{\samepage
\[
\CC_{\infty/2}^d := \{S\in \CC_{\infty/2}\,|\, d(S) = d\}\subset
\CC_{\infty/2}.
\]
For the moment we will be interested in $\CC_{\infty/2}^0$.}

We can get each element of $\CC^0_{\infty/2}$ by starting from the {\it vacuum state, or Dirac sea}{\samepage
\[
S_0 = \{0, 1, \dots \},
\]
where the boxes $0, 1, \dots$ being filled, and then moving some $n$ balls to the left.}

Let us assign to
\[
S = \{a_0, a_1, \dots \}\in \CC^0_{\infty/2}
\]
a partition $\nu = \nu(S)$ by the rule
\begin{gather*}
\nu_i = i - a_i.
\end{gather*}
This way we get a bijection between $\CC^0_{\infty/2}$ and the set of partitions,
cf.~\cite[Lemma~8.1]{SW}.

\begin{Example}\label{example4.2.2}
 If
\[
S = \{-2, -1, 2, 3, \dots\}
\]
then $\lambda(S) = (22)$.
\end{Example}

\subsubsection{From semi-infinite cells to finite ones}
For $n\in \BN$ let
\begin{gather*}
\CC_{\infty/2, n} = \{ S = (a_0, a_1, \dots)\,|\, a_0\geq - n, a_{n} = n\} \subset \CC_{\infty/2};
\end{gather*}
note that $a_n = n$ implies $a_m = m$ for all $m\geq n$.

Let
\[
\CC^0_{\infty/2, n} = \CC_{\infty/2, n} \cap \CC^0_{\infty/2}.
\]

Note that we have a bijection
\begin{gather*}
\CC^{0}_{\infty/2, n} \isom \CC^n_{2n},
\end{gather*}
obvious from the ``balls in boxes'' picture.

Namely, we have inside $\CC^{0}_{\infty/2, n}$ the minimal state
\[
S_\minn = (a_i)\qquad \text{with}\quad a_i = 1\quad \text{for} \ -n \leq i \leq -1 \ \text{and for} \ i\geq n
\]
from which one gets all other states in $\CC^{0}_{\infty/2, n}$ by moving the balls to the right, until we reach the maximal state
\[
S_\maxx = (b_i)\qquad \text{with}\ b_i = 1 \ \text{and for}\ i\geq 0.
\]

For
\[
I\in \CC_{2n}^n\isom \CC_{\infty/2,n}^0\subset \CC^0_{\infty/2}
\]
we will denote by $\nu(I)$ the corresponding partition.

\subsubsection{Transposed cells} \label{section4.2.4}
For $I\in \CC_n^i$ let $I^t = I'\in \CC_n^i$ denote the
``opposite'', or transposed cell which in ``balls and boxes'' picture it is obtained
by reading $I$ from right to left.
The corresponding partition $\lambda(I^t) = \lambda(I)^t$ has the transposed Young diagram.

\subsection{Initial Schur functions and the Wronskian}\label{section4.3}
 Let us introduce new coordinates $t_1, t_2, \dots $ related to~$h_j$ by the formula
\[
e^{\sum_{i=1}^\infty t_iz^i} = 1 + \sum_{j=1}^\infty h_jz^j,
\]
cf.\ \cite[equation~(8.4)]{SW}.

For example
\[
h_1 = t_1, \qquad h_2 = t_2 + \frac{t_1^2}{2},
\]
etc.

Let us consider the Schur functions $s_\nu$ as functions of $t_i$. The first coordinate
$x = t_1$ is called {\em the space variable}, whereas $t_i$, $i\geq 2$, are ``the times''.

We will be interested in the ``initial'' Schur functions, the values of $s_\nu(t)$ for
$t_2 = t_3 = \dots = 0$.

By definition,
\begin{gather}
h_i(x,0, 0, \dots) = \frac{x^i}{i!}.
\label{eq4.3.1}
\end{gather}

Let us return to the unit Wronskian.
It is convenient to consider a limit $\BZ_{> 0}\times \BZ_{> 0}$ Wronskian matrix
\begin{gather*}
\CW_\infty(x) = \lim_{n\ra\infty} \CW_n(x) = \CW\big(1, x, x^2/2, \dots\big)
= \left(\begin{matrix}
1 & x & x^2/2 & x^3/6 & \dots\\
0 & 1 & x & x^2/2 & \dots\\
0 & 0 & 1 & x & \dots\\
\dots & \dots &\dots &\dots &\dots\\
\end{matrix}\right).
\end{gather*}
Let
\[
\nu\colon \ \nu_0\geq \dots \geq \nu_r > 0
\]
be a partition, $S = S(\nu) = (a_i)\in\CC_{\infty/2}^0$; define
\begin{gather*}
\ell(S) = \sum_{i\geq 0} (i - a_i) = \sum_{i\geq 0} \nu_i,
\end{gather*}
cf.\ \cite[a formula after Proposition~2.6]{SW},
and ``the hook factor''
\begin{gather*}
\fh(S) = \fh(\nu) = \frac{\prod_{0\leq i<j\leq r} (a_j - a_i)}{\prod_{0\leq i\leq r} (r - a_i)!},
\end{gather*}
cf.\ \cite[Section~I.1, Example~1, equation~(4)]{M}.

On the other hand suppose that $\nu = \nu(I)$ for some $I = I(\nu)\in\CC_{2n}^n$, cf.\ Claim~\ref{claim3.2.3}.

\begin{Claim}\label{claim4.3.1}
\begin{alignat*}{3}
& (i) \quad && s_\nu(x, 0, 0, \dots) = \fh(S)x^{\ell(S)}, &\\
& (ii)\quad && \Delta_{I^t}(\CW_\infty(x)) = \fh(S) x^{\ell(S)}.
\end{alignat*}
\end{Claim}

\begin{proof} (i) is \cite[proof of Proposition~8.6]{SW}.
(ii) is a consequence of a more general Claim~\ref{claim4.5.0} below.
\end{proof}

\begin{Example}\label{example4.3.2}
 Let
\[
I = (1100),\qquad
S = \{ -2, -1, 2, 3, \dots\},
\]
then
\[
\nu = (22);\qquad \ell(S) = 4,\qquad \fh(S) = \frac{1}{12}.
\]
Then
\begin{alignat*}{3}
& (i)\quad && s_\nu = h_2^2 - h_1h_3 = \frac{t_1^4}{12} - t_1t_3 + t_2^2,&\\
& (ii)\quad && I(\nu) = (1100), \qquad I'(\nu) = (0011),&\\
&&& \Delta_{I'}(\CW_\infty(x)) = \left|\begin{matrix}
x^2/2 & x^3/6\\
x & x^2/2
\end{matrix}\right| = \frac{x^4}{12}.&
\end{alignat*}
\end{Example}

\subsection[Polynomials yn(g)(x) and initial tau-functions: the middle case]{Polynomials $\boldsymbol{y_n(g)(x)}$ and initial tau-functions: the middle case}\label{section4.4}

Let $n \geq 1$. For
\[
I\in \CC_{2n}^n \isom \CC^0_{\infty/2, n}\subset \CC^0_{\infty/2}
\]
let $\nu(I)$ denote the corresponding partition.

Consider the Grassmanian
\[
\Gr_{2n}^n = {\rm GL}_{2n}/P_{n,n}.
\]
For a matrix $g\in {\rm GL}_{2n}$ we define its tau-function $\tau(g)$ which will be a function
of variables $t_1, t_2, \dots, $ by
\begin{gather}
\tau(g)(t) = \tau_n(g) = \sum_{I\in \CC_{2n}^n} \Delta_I(g)s_{\nu(I)}(t),
\label{eq4.4.1}
\end{gather}
cf.\ \cite[Proposition~8.3]{SW}.

Abuse of the notation; better notation: $\tau(\bar g)$, $\bar{g}\in \Gr_{2n}^n$.
This subspace is described below, see Section~\ref{section4.6}.

\begin{Examples}\label{examples4.4.1}
(a) $n = 1$. There are $2$ cells in $\Gr_2^1 = \BP^1$:
\[
(10) \lra (01)
\]
(the arrow indicates the Bruhat order),
which correspond to the following semi-infinite cells of virtual dimension $0$:
\[
(-1,1,2, \dots),\qquad (0,1,2, \dots),
\]
which in turn correspond to partitions
\[
(1), \qquad ()
\]
with Schur functions
\[
s_{(1)} = h_1 = t_1, \qquad s_{()} = 1.
\]
Correspondingly,
for $g\in {\rm GL}_2$, the middle tau-function $\tau_1(g)$ has $2$ summands:
\[
\tau_1(g) = \Delta_1(g)s_{(1)} + \Delta_2(g)s_{()} = a_{11}t_1 + a_{12}
\]
for $g = (a_{ij})$.

{\em Differential equation.}
Suppose for simplicity that $a = a_{12} = 1$, introduce the notation $x = t_1$ for the space variable, so $\tau(g) = 1 + ax$.

Let
\begin{gather*}
u(x) = 2\frac{{\rm d}^2\log\tau(g)}{{\rm d}^2x}.
\end{gather*}
Then
\[
u(x) = - \frac{2a^2}{(1 + ax)^2}.
\]
It satisfies a differential equation
\[
6 u u_x + u_{xxx} = 0,
\]
which is the stationary KdV.

(b) $n = 2$. There are $6$ cells in $\Gr_4^2$:
\[
(1100)\lra \begin{matrix} & (1001) & \\
\nearrow & & \searrow & \\
(1010) & & (0101)\\
\searrow & & \nearrow & \\
 & (0110)
\end{matrix}
\lra (0011)
\]
the arrows indicate the Bruhat, or {\it balls in boxes} order: we see how $1$'s (the balls) are moving to the right to the empty boxes.

They correspond to the following semi-infinite cells of virtual dimension $0$:
\begin{gather*}
(-2, -1, 2, 3, \dots),\qquad (-2, 0, 2, 3, \dots),\qquad (-2, 1, 2, 3, \dots),\\
(-1, 0, 2, 3, \dots),\qquad (-1, 1, 2, 3, \dots),\qquad (0, 1, 2, 3, \dots),
\end{gather*}
which in turn correspond to partitions
\[
(22)\lra \begin{matrix} & (2) & \\
\nearrow & & \searrow & \\
(21) & & (1)\\
\searrow & & \nearrow & \\
 & (11)
\end{matrix}
\lra ()
\]
with Schur functions:
\begin{gather*}
s_{(22)} = h_2^2 - h_3h_1 = \frac{t_1^4}{12} + t_2^2 - t_1t_3,
\qquad
s_{(21)} = h_1h_2 - h_3 = \frac{t_1^3}{3} - t_3,\\
s_{(2)} = h_2 = \frac{t_1^2}{2} + t_2,
\qquad
s_{(11)} = h_1^2 - h_2 = \frac{t_1^2}{2} - t_2,\qquad
s_{(1)} = h_1 = t_1,\qquad s_{()} = 1.
\end{gather*}
Thus for $g\in {\rm GL}_4$, $\tau_2(g)$ has $6$ summands:
\begin{gather}
\tau_2(g) = \Delta_{12}(g)s_{(22)} + \Delta_{13}(g)s_{(21)} + \Delta_{23}(g)s_{(11)} +
 \Delta_{14}(g)s_{(2)} + \Delta_{24}(g)s_{(1)} + \Delta_{34}(g)s_{()}\nonumber\\
\hphantom{\tau_2(g)}{}
= \Delta_{34}(g) + \Delta_{24}(g)t_1 + (\Delta_{14}(g) + \Delta_{23}(g) )\frac{t_1^2}{2} +
 (\Delta_{14}(g) - \Delta_{23}(g) )t_2 \nonumber\\
\hphantom{\tau_2(g)=}{}
+ \Delta_{13}(g)\left(\frac{t_1^3}{3} - t_3\right) + \Delta_{12}(g)
\left(\frac{t_1^4}{12} + t_2^2 - t_1t_3\right),
\label{eq4.4.3}
\end{gather}
where $\Delta_{ij}(g)$ denotes the minor with $i$-th and $j$-th columns.
It depends on $3$ variables $x = t_1$, $y = t_2$, $t = t_3$.
\end{Examples}

We deduce from the above a result from \cite[Lemma~7.5]{VW}:

\begin{Theorem}\label{theorem4.4.2}
If
\[
\fW(g) = (y_1(g), \dots, y_{2n}(g))
\]
then
\[
\tau(g)(x, 0, \dots) = \ty_n(g)(x).
\]
Here we use the reciprocal polynomials $\ty_i(g)$ defined in \eqref{eq3.3.1}.
\end{Theorem}

The initial values of tau-functions $\tau(g)(x, 0, \dots)$ make their appearance
in \cite[Proposition~8.6]{SW}.

\begin{proof} We use the definition \eqref{eq4.4.1}:
\[
\tau_n(g) = \sum_{I\in\CC^n_{2n}}\Delta_I(g)s_{\nu(I)}(t_1, t_2, \dots),
\]
then put $t_2 = t_3 = \dots = 0$ in it:
\[
\tau_n(g)(x, 0, \dots) = \sum_{I\in\CC^n_{2n}}\Delta_I(g)\fh(\nu(I))x^{\fl(I)}
\]
by Claim~\ref{claim4.3.1}.

On the other hand, by Theorem~\ref{theorem3.3}
\[
y_n(g)(x) = \sum_{I\in\CC^n_{2n}}\Delta_I(g)m(I)\frac{x^{\fl(I)}}{\fl(I)!},
\]
whence
\begin{gather*}
\ty_n(g)(x) = \sum_{I\in\CC^n_{2n}}\Delta_I(g)m(I')\frac{x^{\fl(I')}}{\fl(I')!}.\tag*{\qed}
\end{gather*}\renewcommand{\qed}{}
\end{proof}

\begin{Lemma}[hook lemma]\label{lemma4.4.2.1} For all $I$
\[
\fh(\nu(I)) = \frac{m(I')}{\fl(I')!}.
\]
\end{Lemma}

Recall that $m(I)$ is the number of paths from $I_\minn$ to~$I$.
\begin{proof} Induction on the number of cells in the Young diagram of~$\nu$.\end{proof}

\begin{Example}\label{example4.4.2.2}
$n = 4$, $I = I_\minn = (1100)$, $I' = I_\maxx = (0011)$, $\nu(I) = (22)$, $m(I') = 2$, $\fl(I') = 4$,
\[
S(I) = \{-2, -1, 2, 3, \dots\},\qquad \fh(\nu(I)) = \frac{1}{12} = 2\frac{1}{24}.
\]
\end{Example}

Theorem~\ref{theorem4.4.2} follows from the above.

\begin{Example}\label{example4.4.3}
Consider $\tau(g) = \tau_2(g)$ from Example~\ref{examples4.4.1}(b). Putting $t_2 = t_3 = 0$ into
\eqref{eq4.4.3} we get
\begin{gather}
\tau(g)(x,0,0) = \Delta_{34}(g) + \Delta_{24}(g)x + (\Delta_{14}(g) + \Delta_{23}(g) )\frac{x^2}{2}\nonumber\\
\hphantom{\tau(g)(x,0,0) =}{} + \Delta_{13}(g)\frac{x^3}{3} + \Delta_{12}(g)\frac{x^4}{12} = \ty_2(g)(x),
\label{eq4.4.5}
\end{gather}
see \eqref{eq3.4.1}.
\end{Example}

\subsubsection{Stationary solutions}\label{section4.4.4}
Let us return to the formula \eqref{eq4.4.3}. We see therefrom that if $g$ is such that
\[
\Delta_{14}(g) = \Delta_{23}(g),
\]
and
\[
\Delta_{13}(g) = \Delta_{12}(g) = 0,
\]
then $\tau(g)$ does not depend on $t_2$ and $t_3$, and therefore $\tau(g) = y_2(g)$.

Note the Pl\"ucker relation
\[
\Delta_{12}(g)\Delta_{34}(g) - \Delta_{13}(g)\Delta_{24}(g) + \Delta_{14}(g)\Delta_{23}(g)= 0,
\]
which implies $\Delta_{14}(g) = \Delta_{23}(g) = 0$, i.e., the only part which survives
will be
\[
\tau(g)(t_1) = \Delta_{34}(g) + \Delta_{24}(g)t_1.
\]

\subsection{Case of an arbitrary virtual dimension}\label{section4.5} Let $i\leq n$. Consider an embedding
\[
\CC_n^i\hra \CC_{\infty/2}^{d(n,i)},
\]
where $d(n,i)$ is chosen in such a way that $\nu(I_\maxx) = ()$ (this defines $d(n,i)$ uniquely).
Here $\nu(I)$ denotes the partition corresponding to $I\in \CC_n^i$ under the composition
\[
\CC_n^i\hra \CC_{\infty/2}^{d(n,i)}\isom \CC^0_{\infty/2},
\]
where the last isomorphism is a shift $(a_j) \mapsto (a_{j-i})$, and identifying
$\CC^0_{\infty/2}$ with the set of partitions.
More precisely, for any $d$ a sequence
$S = (a_0, a_1, \dots)$ belongs to $\CC^{d}_{\infty/2}$ iff
$a_i = i - d$ for $i \gg 0$.
To such $S$ there corresponds a partition
\begin{gather*}
\lambda(S) =(\lambda_1\geq \lambda_2 \geq \cdots),\qquad \lambda_i = a_i - i + d.
\end{gather*}

\subsubsection{Schur functions and minors}\label{section4.5.0}

\begin{Claim}\label{claim4.5.0}
Consider an upper triangular Toeplitz matrix
\[
T = \left(\begin{matrix}
1 & h_1 & h_2 & \dots\\
0 & 1 & h_1 & \dots\\
0 & 0 & 1 & \dots\\
\dots & \dots &\dots &\dots
\end{matrix}\right).
\]
Let $I \in \CC_n^i$. Then
\[
s_{\nu(I)}(h) = \Delta_{I^t}(T),
\]
where $I^t\in \CC_n^i$ is the transposed cell $($see Section~{\rm \ref{section4.2.4})}.
Note that for any $i$ and any $I \in \CC_n^i$ the function $s_{\nu(I)}(h)$ depends exactly on $h_1, \dots, h_{n-1}$.
\end{Claim}

\begin{Examples}\label{examples4.5.0.1}
(a) $n= 4$, $i = 2$, $I = (1100)$, $I^t = (0011)$,
The corresponding semi-infinite cell is
$S(I) = \{-2, -1, 2, 3, \dots\}$, $d(4, 2) = 0$,
$\nu(I) = (22)$.

(b) $n= 4$, $i = 3$, $I = (1110)$, $I^t = (0111)$ The corresponding semi-infinite cell is
$S(I) = \{0, 1, 2, 4, 5, \dots\}$, $d(4,3) = - 1$, $\nu(I) = (111)$.
\end{Examples}

Afterwards we can apply the same construction as above:
to $g\in {\rm GL}_{2n+i}$ we assign a tau-function
\begin{gather*}
\tau(g)(t) = \sum_{I\in \CC_{2n + i}^n} \Delta_I(g)s_{\nu(I)}(t).
\end{gather*}

We deduce from the above a result from \cite[Lemma~7.5]{VW}:

\begin{Theorem}\label{theorem4.5.2}
Let
\[
\fW(g) = (y_1(g), \dots, y_{2n+i}(g)),
\]
then
\[
\ty_n(g)(x) = \tau(g)(x,0, \dots).
\]
Here $\ty_n$ denotes the reciprocal polynomial, see \eqref{eq3.3.1}.
\end{Theorem}

\begin{Example}[projective spaces]\label{example4.5.3}
 Let $n\geq 1$ be arbitrary.

(a) Let $i = 1$.
There are $n$ cells in $\BP^{n-1} = \Gr_n^1$:
\[
(10\dots 0)\lra (01\dots 0)\lra \cdots (00\dots 1),
\]
which correspond to the semi-infinite cells of virtual dimension $n - 1$:
\[
(10\dots 011\dots)\lra (01\dots 011\dots)\lra \cdots (00\dots 111\dots),
\]
(in BB picture), or
\[
(0,n,n+1, \dots) \lra (1,n,n+1, \dots) \lra (n-1,n,n+1, \dots)
\]
which correspond to the partitions
\[
(n-1) \lra (n-2) \lra \cdots \lra ()
\]
with Schur functions
\[
s_{(i)}(h) = h_i,
\]
whence
\[
s_{(i)}(x, 0, \dots) = \frac{x^i}{i!},
\]
cf.~\eqref{eq4.3.1}.

So for a matrix $g = (a_{ij})\in {\rm GL}_n$ its first tau-function
\[
\tau_1(g)(t_1,\dots, t_{n-1}) = \sum_{j=0}^{n-1} a_{1,j+1}h_{n-j-1},
\]
and
\[
\tau_1(g)(x, 0, \dots) = \sum_{j=0}^{n-1} a_{1,n-j}\frac{x^{j}}{j!} = \ty_1(g),
\]
as it should be, the contraction map {\em Pl\"ucker $\lra$ Wronsky} being the identity.

(b) {\em The dual $($conjugate$)$ space $\big(\BP^{n-1}\big)^\vee$.} Let $i = n - 1$.
There are $n$ cells in $\big(\BP^{n-1}\big)^\vee = \Gr_n^{n-1}$:
\[
(11\dots 110)\lra \cdots \lra (101\dots 1) \lra (01\dots 1),
\]
which correspond to the semi-infinite cells of virtual dimension $- 1$:
\[
(11\dots 1011\dots)\lra \cdots \lra (101111\dots)\lra (011\dots)
\]
(in BB picture), or to sequences $S$
\[
(0,1,\dots, n-1,n+1, \dots)\lra \cdots \lra (0,1,3,4,\dots) \lra (0,2,3, \dots),
\]
which correspond to the partitions
\[
\big(1^{n-1}\big) \lra \cdots \lra (1) \lra ()
\]
with Schur functions
\[
s_{(1^i)} = e_i
\]
(an elementary symmetric function; see \cite[Chapter~I, equation~(3.8)]{M} for the relation between
Schur functions of the conjugate partitions).
Whence
\[
s_{(1^i)}(x, 0, \dots) = \frac{x^i}{i!}.
\]

For a matrix $g = (a_{ij})\in {\rm GL}_n$ its $(n-1)$-th tau-function
\[
\tau_{n-1}(g)(t_1,\dots, t_{n-1}) = \sum_{j=1}^{n} \Delta_{[1\dots \hat{n-j}\dots n]}(g)e_j\frac{x^{j-1}}{(j-1)!}
\]
(the coefficients being $(n-1)\times (n-1)$-minors),
so
\[
\tau_{n-1}(g)(x, 0, \dots) = \sum_{j=1}^{n} \Delta_{[1\dots \hat{n-j}\dots n]}(g)\frac{x^{j-1}}{(j-1)!} = \ty_{n-1}(g).
\]
\end{Example}

\subsection{Wronskians as tau-functions}\label{section4.6}
 One can express the above as follows. Consider a {\it Tate} vector space of Laurent power series
\[
H = \bk((z)) = \bigg\{\sum_{i\geq j} a_iz^i\,|\, j\in \BZ,\, a_i\in \bk\bigg\}.
\]
It is equipped with two subspaces, $H_+ = \bk[[z]]$ and
\[
H_- = \bigg\{\sum_{- j \leq i\leq 0} a_iz^i, \, j\in \BN\bigg\}.
\]
Let $\Gr = \Gr_\infty^{\infty/2}$ denote the Grassmanian of subspaces $L\subset H$ of the form
\[
L = L_0 + z^kH_+,\qquad k\in \BN,
\]
where $L_0 = \langle f_1, \dots, f_q\rangle $ is generated by a finite number of Laurent {\it polynomials}
$f_i(z)\in \bk\big[z, z^{-1}\big]$, cf.\ \cite[Section~8]{SW}.

In other words, such $L$ should admit a {\it topological} base of the form
\[
\big\{ f_1(z), \dots, f_q(z), z^k, z^{k+1}, \dots\,|\, f_i(z)\in \bk\big[z, z^{-1}\big]\big\}.
\]

 For example, in Section~\ref{section3.5.3} above we see a description of embeddings
\[
\BP^{n-1} = \Gr_n^1\hra \Gr,\qquad
\big(\BP^{n-1}\big)^\vee = \Gr_n^{n-1}\hra \Gr.
\]
To each $L\in \Gr$ there corresponds a tau-function
\[
\tau_L(t) = \tau_L(t_1, t_2, \dots),
\]
cf.\ \cite[Proposition 8.3]{SW}.

Given a sequence of polynomials
\[
\ff = (f_1(z), \dots, f_m(z)) \subset \bk[[z]]_{\leq d}
\]
of degree $\leq d$, $m\leq d$, let us associate to it a subspace
\[
L(\ff) = \big\langle f_1\big(z^{-1}\big), \dots, f_m\big(z^{-1}\big), z, z^{2}, \dots \big\rangle \subset H
\]
belonging to $\Gr$.

\begin{Theorem}\label{theorem4.6.1}
\[
\tau_{L(\ff)}(z,0, \dots) = W(\ff).
\]
\end{Theorem}

This statement is a reformulation of Theorems~\ref{theorem4.4.2} and \ref{theorem4.5.2}.

\subsection{Full flags and MKP}\label{section4.7} For $n\in \BZ_{\geq 2}$ define a {\it semi-infinite flag
space} $\Fl_n^{\infty/2}$ whose elements are sequences of subspaces
\[
L_1\subset \dots \subset L_{n-1}\subset H = \bk((z)),\qquad L_i\in \Gr\subset \Gr_\infty^{\infty/2},\qquad
\dim L_i/L_{i-1} = 1.
\]
This is a subspace of the semi-infinite flag space considered in \cite[Section~8]{KP}, whose elements
parametrize the rational solutions of the modified Kadomtsev--Petviashvili hierarchy.

To each flag $F = (L_1\subset \dots \subset L_{n-1})\in \Fl_n^{\infty/2}$ we assign
its tau-function which is by definition a collection of $n - 1$ grassmanian tau-functions:
\[
\tau_F(t) := (\tau_{L_1}(t), \dots, \tau_{L_{n-1}}(t)),\qquad t = (t_1, t_2, \dots).
\]

Let $g = (a_{ij})\in {\rm GL}_n(\bk)$, we associate to it as above $n$ polynomials
\[
b_i(z) = \sum_{j=0}^{n-1} a_{i,j+1}z^j,\qquad 1\leq i \leq n.
\]
We denote
\[
L_i(g) := L(b_1(z), \dots, b_i(z))\in \Gr.
\]
Taking all $1\leq i\leq n - 1$ we get a flag
\[
F(g) = (L_1(g)\subset \dots \subset L_{n-1}(g))\in \Fl^{\infty/2}_n
\]
and can consider the corresponding tau-function
\[
\tau_{F(g)}(t) = (\tau_{L_1(g)}(t), \dots, \tau_{L_{n-1}(g)}(t)).
\]

\begin{Claim}\label{claim4.7.1}
For $g\in {\rm GL}_n(\bk)$ consider its image under the Wronskian map
\[
\fW(g) = (y_1(g)(x), \dots, y_n(g)(x)).
\]
Then
\[
\tilde\fW(g) := (\ty_1(g)(x), \dots, \ty_n(g)(x))
\]
coincides with the initial value $\tau_{F(g)}(x, 0,\dots)$ of the tau-function $\tau_{F(g)}$ of the MKP hierarchy
corresponding to the semi-infinite flag $F(g)$.
\end{Claim}

This is an immediate consequence of the Grassmanian case.

\begin{Example}\label{example4.7.1}
Let $n = 4$. Given $g = (a_{ij})\in {\rm GL}_4(\bk)$, we associate to it
four Laurent polynomials encoding its rows
\[
b_i(z) = a_{i1}z^{-2} + \dots + a_{i4}z,\qquad 1\leq i \leq 4,
\]
and a semi-infinite flag
\[
F(g) = (L_1(g)\subset L_2(g) \subset L_3(g))\in \Fl_4^{\infty/2},
\]
where
\begin{gather*}
L_1(g) = \langle g_1(z)\rangle + z^2H_+, \qquad L_2(g) = \langle g_1(z), g_2(z)\rangle + z^2H_+,\\
L_3(g) = \langle g_1(z), g_2(z), g_3(z)\rangle + z^2H_+.
\end{gather*}

We have
\begin{gather*}
\tau_1(g)(t_1, t_2, t_3) = \sum_{j=0}^3 a_{1,j+1}h_{3-j} = a_{14} + a_{13}t_1 + a_{12}
\left(\frac{t_1^2}{2} + t_2 \right) + a_{11}\left(\frac{t_1^3}{6} + t_1t_2 + t_3\right),\\
\ty_1(g)(x) = \sum_{j=0}^3 a_{1,4-j}\frac{x^j}{j!},
\end{gather*}
see Example~\ref{example4.5.3}(a);
\begin{gather*}
\tau_2(g)(t_1, t_2, t_3) = \Delta_{34}(g) + \Delta_{24}(g)t_1 + (\Delta_{14}(g) + \Delta_{23}(g) )\frac{t_1^2}{2} +
 (\Delta_{14}(g) - \Delta_{23}(g) )t_2\\
\hphantom{\tau_2(g)(t_1, t_2, t_3) =}{} + \Delta_{13}(g)\left(\frac{t_1^3}{3} - t_3\right) + \Delta_{12}(g)
\left(\frac{t_1^4}{12} + t_2^2 - t_1t_3\right),\\
\ty_2(g)(x) = \Delta_{34}(g) + \Delta_{24}(g)x +(\Delta_{14}(g) + \Delta_{23}(g))\frac{x^2}{2} +
 \Delta_{13}(g)\frac{x^3}{3} + \Delta_{12}(g)\frac{x^4}{12},
\end{gather*}
see \eqref{eq4.4.3}, \eqref{eq4.4.5};
\begin{gather*}
\tau_3(g)(t_1, t_2, t_3) = \Delta_{123}(g) + \Delta_{124}(g)t_1 + \Delta_{134}(g)
 \left(\frac{t_1^2}{2} - t_2\right) + \Delta_{234}(g)\left(\frac{t_1^3}{6} - t_1t_2 + t_3\right),\\
\ty_3(g)(x) = \Delta_{123}(g) + \Delta_{124}(g)x + \Delta_{134}(g)\frac{x^2}{2} + \Delta_{234}(g)\frac{x^3}{6},
\end{gather*}
see Example~\ref{example4.5.3}(b).
\end{Example}

\section[W5 and Desnanot-Jacobi identities]{$\boldsymbol{W5}$ and Desnanot--Jacobi identities}\label{section5}

\subsection[W5 identity]{$\boldsymbol{W5}$ identity} Recall a formula from \cite[Section~2.1]{SV}.
Let
\[
\ff = (f_1(x), f_2(x), \dots)
\]
be a sequence of functions. For a totally ordered finite subset
\[
A = \{i_1, \dots, i_a\}\subset \BZ_{>0}
\]
we denote
\[
\ff_A = (f_{i_1}, \dots, f_{i_a}), \qquad W(A) := W(\ff_A).
\]
We write $[a] = \{1, \dots, a\}$.

\begin{Lemma}\label{lemma5.1.1} Let
\[
A = [a+1], \qquad B = [a]\cup \{a + 2\}.
\]
Then
\begin{gather}
W(W(A), W(B)) = W(A\cap B)\cdot W(A\cup B).
\label{eq5.1.1}
\end{gather}
\end{Lemma}

This is a particular case of \cite[Section~9]{MV}.

\begin{Example}\label{example5.1.2}
\[
W(W(f_1,f_2), W(f_1,f_3)) = f_1W(f_1,f_2,f_3).
\]
\end{Example}

\subsection{Desnanot--Jacobi, aka Lewis Carrol}\label{section5.2} Let $A$ be an $n\times n$ matrix.
Denote by $A_{1n,1n}$ the $(n - 2)\times (n - 2)$ submatrix obtained from $A$ by deleting the first and the $n$-th row
and the first and the $n$-th column, etc.
Then
\begin{gather}
\det A\det A_{1n,1n} = \det A_{1,1}\det A_{n,n} - \det A_{1,n}\det A_{n,1}.
\label{eq5.2.1}
\end{gather}

Let us rewrite \eqref{eq5.1.1} in the form
\[
W(A\cap B)\cdot W(A\cup B) = W(A)W(B)' - W(A)'W(B);
\]
we see that it resembles \eqref{eq5.2.1}.
And indeed, if we apply \eqref{eq5.2.1} to the Wronskian matrix $\CW(A\cup B)$,
we get \eqref{eq5.1.1}. This remark appears in \cite[the beginning of Section~3]{C}.
We leave the details to the reader.

\subsection{Example: Lusztig vs Wronskian mutations} \label{section5.4}
Recall the situation \eqref{eq1.2}, where we suppose that $m = n$.
The main result of \cite{SV} describes the compatibility of the map $\fW$ with multiplication by an {\it upper} triangular
matrix
\[
e_{i,i+1}(c) = I_n + ce_{ij},\qquad e_{ij} = (\delta_{pi}\delta_{qj})_{p,q}\in {\rm GL}_n(\bk).
\]

Let $M\in {\rm GL}_n(\bk)$,
\begin{gather*}
\fb(M) = (b_1(M)(x), b_2(M)(x), b_3(M)(x), \dots),\\
\fW(M) = (y_1(M)(x), y_2(M)(x), y_3(M)(x), \dots).
\end{gather*}

The following example is taken from \cite[proof of Theorem~4.4]{SV}.
We have
\[
\fW(e_{23}M) = (y_1(e_{23}M), y_2(e_{23}M), \dots),
\]
where
\[
\fb(e_{23}M) = (b_1(M), b_2(M) + b_3(M), \dots).
\]
Thus
\[
y_1(e_{23}(c)M) = y_1(M),
\]
whereas
\[
y_2(e_{23}(c)M) = y_2(M) + cW(b_1(M), b_3(M)).
\]
The $W5$ identity implies
\begin{gather*}
W(y_2(M), y_2(e_{23}(c)M)) = cW(y_2(M),W(b_1(M), b_3(M)))\\
\hphantom{W(y_2(M), y_2(e_{23}(c)M))}{} =
cW(W(b_1(M), b_2(M),W(b_1(M), b_3(M)))\\
\hphantom{W(y_2(M), y_2(e_{23}(c)M))}{} =
c b_1W(b_1(M), b_2(M), b_3(M)) =c y_1(M)y_3(M).
\end{gather*}
The resulting equation
\begin{gather*}
W(y_2(M), y_2(e_{23}(c)M)) = c y_1(M)y_3(M)
\end{gather*}
is called a {\it Wronskian mutation} equation. It is regarded in~\cite{SV} as a differential equation
of the first order on an unknown function $\ty_2(c, M) = y_2(e_{23}(c)M)$;
together with some initial conditions it determines the polynomial $\ty_2(c, M)$ uniquely from given $y_i(M)$, $1\leq i\leq 3$.

This equation is a modification of a similar differential equation from \cite{MV} whose solution is a~function
\[
\ty'_2(c, M) = y_2(e'_{23}(c)M),
\]
where
\[
e'_{23}(c) = ce_{23}\big(c^{-1}\big) = cI_n + e_{23}.
\]

\subsection*{Acknowledgements}
We are grateful to A.~Kuznetsov for a useful discussion, to V.~Kac for sending us the paper \cite{KP},
and to anonymous referees for correcting some inaccuracies and drawing our attention to~\cite{EG}.
V.G.~has been partially supported by the HSE University Basic Research Program, Russian Academic Excellence Project 5-100,
and by the RSF Grant No.~20-61-46005.

\pdfbookmark[1]{References}{ref}
\LastPageEnding

\end{document}